\documentclass{amsart}

\usepackage{geometry}

\usepackage{amsthm}
\usepackage{amsmath}
\usepackage{amssymb}
\usepackage{comment}
\usepackage{bm}
\usepackage{enumerate}
\usepackage{mathtools}
\usepackage{cite}
\usepackage{enumitem}

\newtheorem{theorem}{Theorem}

\newtheorem{corollary}{Corollary}

\newtheorem{lemma}{Lemma}

\title{On elementary estimates for the partition function}

\author{MIZUKI AKENO}
\email{akeno.mizuki.tkb\_ee@u.tsukuba.ac.jp}
\date{\today}
\address{Department of Mathematics, University of Tsukuba, Tsukuba, Japan}
\keywords{Asymptotic formulas, partition functions, generating functions}
\subjclass[2020]{11P21, 11P81}
\begin{document}

\begin{abstract}
In this paper, we obtain upper and lower bounds for the partition function $p(n)$ by using an elementary geometric inequality in Euclidean space, and we extend the method to generalizations of the partition function.
\end{abstract}

\maketitle

\section{Introduction}
Let $p(n)$ denote the partition function, i.e. the number of partitions of the nonnegative integer $n$. 
The generating function of $p(n)$ is
\begin{equation} \sum_{n=0}^{\infty} p(n) z^n = \prod_{n=1}^{\infty} (1-z^{n})^{-1}. \end{equation}
Using Cauchy's residue theorem together with an asymptotic expansion of the generating function, Hardy and Ramanujan \cite{hardy-ramanujan} proved
\begin{equation} p(n) = \frac{e^{\pi\sqrt{\frac{2}{3}n}}}{4n\sqrt{3}} (1+o(1)) = \frac{2\pi}{(24n)^{\frac{3}{4}}}I_{0}{(2\sqrt{\zeta{(2)} n})}(1+o(1)), \quad (n \to \infty) \end{equation}
where $I_{\alpha}(z)$ is the modified Bessel function of the first kind defined for $z \in \mathbb{C}$ by 
\begin{equation} I_{\alpha}(z) = \sum_{n=0}^{\infty} \frac{1}{n! \Gamma{(n + \alpha + 1)}} \left( \frac{z}{2} \right)^{2n+\alpha}. \end{equation}
The asymptotic behavior of $I_{\alpha}$ is known to be
\begin{equation} \label{Bessel asy} I_{\alpha}(x) = \frac{e^x}{\sqrt{2 \pi x}}(1 + o(1)), \quad (x \to \infty). \end{equation}
In this paper, we obtain the following explicit upper and lower bounds for $\sum_{n \leq N} p(n)$ and its variants using elementary methods. 
\begin{theorem}\label{P1}
For all $N \in \mathbb{Z}_{\geq 1}$, we have
\begin{equation} \label{goal1} e^{-H_N} I_0 (2 \sqrt{\zeta_{N}{(2)}N}) \leq \sum_{n=0}^{N} p (n) \leq I_0 (2 \sqrt{\zeta_{N}{(2)}N}), \end{equation}
\end{theorem}
Here, $\zeta_{N}{(s)}$ denotes the truncated zeta function $\sum_{n=1}^{N} \frac{1}{n^s}$, $H_N$ denotes the harmonic number $\sum_{n=1}^{N} \frac{1}{n}$. 

Classical bounds for $p(n)$ can be found, for example, in Sections 2--3 of \cite{hardy-ramanujan} and Theorem 15.7 of \cite{van2001course}. 
These bounds rely on analytic properties of the generating function of $p(n)$. 
Our method depends only on combinatorial arguments and lattice-point/volume estimates for counting solutions to linear inequalities. 
A stronger estimate was independently obtained by Kane \cite{kane2006elementary}, whose method is similar in spirit to ours but technically more elaborate.

The strength of our approach is its flexibility. 
We obtain two generalizations of \eqref{goal1}, that include applications to bounds on well-known generalizations of the partition function, such as $q$-th power partition \cite{wrightIII} and plane partition \cite{wrightI}. 
As a simple illustration, for each $X \geq 1$ we obtain explicit upper and lower bounds on the number of nonnegative integer solutions $(n_i)$ satisfying
\[ n_1 + n_2 \sqrt{2} + n_3 \sqrt{3} + n_4 \sqrt{4} + \cdots \leq X. \]

\section{Bounds for partition function}

\subsection{Sum of the partition function}

We begin by giving a short proof of Theorem \ref{P1}.

All infinite sums and products in $\mathbb{R}[[z]]$ are interpreted coefficientwise, that is, they are defined by coefficientwise limits. 
We define 
\[ e^{f} = \sum_{n=0}^{\infty} \frac{f^n}{n!} \]
for all $f \in \mathbb{R}[[z]]$. 
Note that $\sum_{n=0}^{\infty} [z^k] f^n/n!$ converges absolutely for each $k$. 
Moreover, $e^{f+g} = e^{f} e^{g}$ for all $f,g \in \mathbb{R}[[z]]$.

\begin{proof}[Proof of Theorem \ref{P1}]

Let $f$ be given by
\begin{equation*} f(z) = \prod_{n=1}^{\infty} {\left(1- z^{n} \right)}^{-1} = \sum_{n=0}^{\infty} p {(n)} z^n. \end{equation*}
Let $w_k=w_k(z) = \sum_{n=1}^{\infty} z^{k n}$. 
For $z \in \mathbb{C}$ with $|z|<1$, one has
\begin{align*}
f(z) &= \exp{\left( \sum_{n=1}^{\infty} -\ln{(1-z^{n})} \right)}, \\
&= \exp{\left( \sum_{n=1}^{\infty} \sum_{k=1}^{\infty} \frac{z^{n k}}{k} \right)}, \\
&= \exp{\left( \sum_{k=1}^{\infty} \frac{w_k}{k} \right)}.
\end{align*}
Hence the same identity holds in $\mathbb{R}[[z]]$. Since $[z^n]w_k=0$ for all $n \leq N$ if $k>N$, we can express the partial sum of the partition function in the form
\[ \sum_{n=0}^{N} p (n) = I[f(z)](N) = I\left[\exp{\left(\sum_{k=1}^{\infty} \frac{w_k}{k}\right)}\right](N) = I\left[\exp{\left(\sum_{k=1}^{N} \frac{w_k}{k}\right)}\right](N), \]
where $I[\cdot](N)$ denotes the sum of coefficients of terms of degree $\leq N$. 
Expanding the exponential, we have 
\[ \exp{\left(\sum_{k=1}^{N} \frac{w_k}{k}\right)} = \prod_{k=1}^{N} \left( \sum_{n=0}^{\infty} \frac{1}{n!} \left(\frac{w_k}{k}\right)^n \right) = \sum_{\bm{r} \in \mathbb{Z}_{\geq 0}^N} c_{\bm{r}} \prod_{k=1}^{N} w^{r_k}_k \]
for some non-negative coefficients $(c_{\bm{r}})_{\bm{r}}$. We also have
\[ e^{H_N} \exp{\left(\sum_{k=1}^{N} \frac{w_k}{k}\right)} = \exp{\left(\sum_{k=1}^{N} \frac{1+w_k}{k}\right)} = \sum_{\bm{r} \in \mathbb{Z}_{\geq 0}^N} c_{\bm{r}} \prod_{k=1}^{N} (1+w_k)^{r_k}. \]
We therefore obtain
\begin{equation} \label{P_interc} \sum_{n=0}^{N} p (n) = \sum_{\bm{r} \in \mathbb{Z}_{\geq 0}^N} c_{\bm{r}} I \left[ \prod_{k=1}^{N} w^{r_k}_k \right] (N), \quad e^{H_N} \sum_{n=0}^{N} p (n) = \sum_{\bm{r} \in \mathbb{Z}_{\geq 0}^N} c_{\bm{r}} I \left[ \prod_{k=1}^{N} (1+w_k)^{r_k} \right] (N). \end{equation}

We now fix $r_1,\ldots,r_N \in \mathbb{Z}_{\geq 0}, r_1+\cdots+r_N=r \geq 1$ and observe that
\[ I\left[ \prod_{k=1}^{N} w_k^{r_k} \right](N) = \#\{ (x^{(k)}_{i})_{\substack{1 \leq k \leq N \\ 1 \leq i \leq r_k}} \in \mathbb{Z}_{\geq 1}^r: \sum_{1 \leq k \leq N} \sum_{1\leq i \leq r_k} k x^{(k)}_{i} \leq N \}, \]
\[ I\left[ \prod_{k=1}^{N} (1+w_k)^{r_k} \right](N) = \#\{ (x^{(k)}_{i})_{\substack{1 \leq k \leq N \\ 1 \leq i \leq r_k}} \in \mathbb{Z}_{\geq 0}^r: \sum_{1 \leq k \leq N} \sum_{1\leq i \leq r_k} k x^{(k)}_{i} \leq N \}. \]
Using the relation between the volume of a subset of $\mathbb{R}^{r}$ and the number of lattice points inside it, we have
\begin{equation} \label{P_ineq} I\left[ \prod_{k=1}^{N} w_k^{r_k} \right](N) \leq V \leq I\left[ \prod_{k=1}^{N} (1+w_k)^{r_k} \right](N) \end{equation}
where
\[ V = \mathrm{vol}\left( \{ (x^{(k)}_{i})_{\substack{1 \leq k \leq N \\ 1 \leq i \leq r_k}} \in \mathbb{R}^r_{\geq 0} : \sum_{1 \leq k \leq N} \sum_{1\leq i \leq r_k} k x^{(k)}_{i} \leq N \} \right) = \frac{1}{r!} \prod_{k=1}^{N} \left( \frac{N}{k} \right)^{r_k}. \]

For all $M \in \mathbb{N}$, one has
\begin{align*}
\sum_{\substack{\bm{r} \in \mathbb{Z}_{\geq 0}^N \\ r_1+ \cdots + r_N \leq M}} c_{\bm{r}} \frac{1}{(r_1+\cdots+r_N)!} \prod_{k=1}^{N} \left( \frac{N}{k} \right)^{r_k} &= \sum_{r=0}^{M} \frac{1}{r!} \sum_{\substack{\bm{r} \in \mathbb{Z}_{\geq 0}^N \\ r_1+ \cdots + r_N = r}} c_{\bm{r}} \prod_{k=1}^{N} \left( \frac{N}{k} \right)^{r_k} \\
&= \sum_{r=0}^{M} \frac{1}{r!} [x^r] \exp\left( \sum_{k=1}^{N} \frac{1}{k} \cdot \frac{xN}{k} \right) \\
&= \sum_{r=0}^{M} \frac{1}{r!} \frac{(N\zeta_N(2))^r}{r!}. \\
\end{align*}
This and \eqref{P_ineq} give
\[ \sum_{\substack{\bm{r} \in \mathbb{Z}_{\geq 0}^N \\ r_1+ \cdots + r_N \leq M}} c_{\bm{r}} I \left[ \prod_{k=1}^{N} w^{r_k}_k \right](N) \leq \sum_{r=0}^{M} \frac{(N\zeta_N(2))^r}{(r!)^2} \leq \sum_{\substack{\bm{r} \in \mathbb{Z}_{\geq 0}^N \\ r_1+ \cdots + r_N \leq M}} c_{\bm{r}} I \left[ \prod_{k=1}^{N} (1+w_k)^{r_k} \right](N). \]
Letting $M \to \infty$ and recalling \eqref{P_interc} completes the proof. 

\end{proof}

\subsection{Framework}

Let
\[ \mathbb{R}\{x_1,\ldots,x_n\}_{\infty} = \{ \sum_{\bm{r} \in \mathbb{Z}^n_{\geq 0}} c_{\bm{r}} \prod_{k=1}^{n} x^{r_k}_k \in \mathbb{R}[[x_1,\ldots,x_n]] \mid \forall x \in (0,\infty), \sum_{\bm{r} \in \mathbb{Z}^n_{\geq 0}} |c_{\bm{r}}| x^{r_1+\cdots+r_n} < \infty \}. \]
Let $\mathfrak{I}_n$ denote the set of linear operators $I:\mathbb{R}\{x_1,\ldots,x_n\}_{\infty} \to \mathbb{R}$ that satisfy the following properties:
\begin{itemize}
\item There exists $C>0$ and
\[ \left| I\left[\prod_{k=1}^{n} x^{r_k}_k\right] \right| \leq C^{r_1+\cdots+r_n} \]
holds for all $r_1,\ldots,r_n \in \mathbb{Z}_{\geq 0}$.  
\item For all $\sum_{\bm{r} \in \mathbb{Z}^n_{\geq 0}} c_{\bm{r}} \prod_{k=1}^{n} x^{r_k}_k \in \mathbb{R}\{x_1,\ldots,x_n\}_{\infty}$, one has
\[ \left| I\left[\sum_{\bm{r} \in \mathbb{Z}^n_{\geq 0}} c_{\bm{r}} \prod_{k=1}^{n} x^{r_k}_k \right] \right| \leq \sum_{\bm{r} \in \mathbb{Z}^n_{\geq 0}} |c_{\bm{r}}| \left| I\left[\prod_{k=1}^{n} x^{r_k}_k\right] \right|. \]
\end{itemize}

\begin{lemma}\label{L_I}
Let $M \in \mathbb{Z}_{\geq 1}, I,I' \in \mathfrak{I}_M$, and let $\sum_{\bm{r} \in \mathbb{Z}^M_{\geq 0}} c_{\bm{r}} \prod_{k=1}^{M} x^{r_k}_k \in \mathbb{R}\{x_1,\ldots,x_M\}_{\infty}$, where $c_{\bm{r}} \geq 0$ for all $\bm{r}$. 
If
\[ 0 \leq I\left[ \prod_{k=1}^{M} x^{r_k}_k \right] \leq I'\left[ \prod_{k=1}^{M} x^{r_k}_k \right] \]
holds for all $r_1,\ldots,r_M \in \mathbb{Z}_{\geq 0}$, then
\[ I\left[\sum_{\bm{r} \in \mathbb{Z}^M_{\geq 0}} c_{\bm{r}} \prod_{k=1}^{M} x^{r_k}_k \right] \leq I'\left[\sum_{\bm{r} \in \mathbb{Z}^M_{\geq 0}} c_{\bm{r}} \prod_{k=1}^{M} x^{r_k}_k \right]. \]
Let $(a_k) \in \mathbb{R}^M_{\geq 0}$. If 
\[ 0 \leq I'\left[ \prod_{k=1}^{M} x^{r_k}_k \right] \leq I\left[ \prod_{k=1}^{M} (a_k+x_k)^{r_k}\right] \]
holds for all $r_1,\ldots,r_M \in \mathbb{Z}_{\geq 0}$, then
\[ I'\left[\sum_{\bm{r} \in \mathbb{Z}^M_{\geq 0}} c_{\bm{r}} \prod_{k=1}^{M} x^{r_k}_k \right] \leq I\left[\sum_{\bm{r} \in \mathbb{Z}^M_{\geq 0}} c_{\bm{r}} \prod_{k=1}^{M} (a_k+x_k)^{r_k} \right]. \]
\end{lemma}

\begin{proof}

The function $\sum_{\bm{r} \in \mathbb{Z}^M_{\geq 0}} c_{\bm{r}} \prod_{k=1}^{M} x^{r_k}_k$ is holomorphic on $(x_i) \in \mathbb{C}^M$. Thus we see that
\[ \sum_{\bm{r} \in \mathbb{Z}^M_{\geq 0}} c_{\bm{r}} \prod_{k=1}^{M} (a_k+x_k)^{r_k} = \sum_{\bm{r} \in \mathbb{Z}^M_{\geq 0}} d_{\bm{r}} \prod_{k=1}^{M} x^{r_k}_k \in \mathbb{R}\{x_1,\ldots,x_M\}_{\infty} \]
for some $(d_{\bm{r}})_{\bm{r}}$ satisfying $d_{\bm{r}}\geq 0$. 

To see the final inequality, it suffices to show that
\begin{equation} I\left[\sum_{\bm{r} \in \mathbb{Z}^M_{\geq 0}} c_{\bm{r}} \prod_{k=1}^{M} x^{r_k}_k \right] = \sum_{\bm{r} \in \mathbb{Z}^M_{\geq 0}} c_{\bm{r}} I\left[\prod_{k=1}^{M} x^{r_k}_k\right] \label{P_I1}, \end{equation}
\begin{equation} I\left[\sum_{\bm{r} \in \mathbb{Z}^M_{\geq 0}} c_{\bm{r}} \prod_{k=1}^{M} (a_k+x_k)^{r_k} \right] = \sum_{\bm{r} \in \mathbb{Z}^M_{\geq 0}} c_{\bm{r}} I\left[\prod_{k=1}^{M} (a_k+x_k)^{r_k} \right] \label{P_I2}. \end{equation}
The equation \eqref{P_I1} follows from the definition of $\mathfrak{I}_M$. 
By \eqref{P_I1} we have
\[ I\left[\sum_{\bm{r} \in \mathbb{Z}^M_{\geq 0}} c_{\bm{r}} \prod_{k=1}^{M} (a_k+x_k)^{r_k} \right] = \sum_{\bm{r} \in \mathbb{Z}^M_{\geq 0}} d_{\bm{r}} I\left[\prod_{k=1}^{M} x_k^{r_k} \right]. \]
Comparing this with the expansion of $\prod_{k=1}^{M} (a_k+x_k)^{r_k}$ in the right hand side of \eqref{P_I2}, gives \eqref{P_I2}. 
\end{proof}

We define $I(\cdot)(N),I'(\cdot)(N):\mathbb{R}\{w_1,\ldots,w_N\}_{\infty} \to \mathbb{R}$ by
\[ I\left[\sum_{\bm{r} \in \mathbb{Z}^N_{\geq 0}} c_{\bm{r}} \prod_{k=1}^{N} w^{r_k}_k \right](N) = \sum_{n=0}^{N} [z^n] \sum_{\bm{r} \in \mathbb{Z}^N_{\geq 0}} c_{\bm{r}} \prod_{k=1}^{N} \left( \sum_{n=1}^{\infty} z^{nk} \right)^{r_k}, \]
\[ I'\left[\sum_{\bm{r} \in \mathbb{Z}^N_{\geq 0}} c_{\bm{r}} \prod_{k=1}^{N} w^{r_k}_k \right](N) = \sum_{\bm{r} \in \mathbb{Z}^N_{\geq 0}} c_{\bm{r}} \frac{1}{(r_1+\cdots+r_N)!} \prod_{k=1}^{N} \left( \frac{N}{k} \right)^{r_k}. \]
Clearly $I' \in \mathfrak{I}_N$. We also have $I \in \mathfrak{I}_N$, by the following lemma.

\begin{lemma}\label{L_Id}
Let $M,N \in \mathbb{Z}_{\geq 1}$, $w_k(z) \in \mathbb{R}[[z]], k=1,\ldots,M$. Let $I:\mathbb{R}\{w_1,\ldots,w_M\}_{\infty} \to \mathbb{R}$ be given by
\[ I\left[\sum_{\bm{r} \in \mathbb{Z}^M_{\geq 0}} c_{\bm{r}} \prod_{k=1}^{M} w^{r_k}_k \right] = \sum_{n=0}^{N} [z^n] \sum_{\bm{r} \in \mathbb{Z}^M_{\geq 0}} c_{\bm{r}} \prod_{k=1}^{M} w^{r_k}_k(z). \]
Then $I \in \mathfrak{I}_M$. 
\end{lemma}

\begin{proof}
Let $w_k(z) = \sum_{n=0}^{\infty} a_{n,k} z^n$ and set $A=\max_{\substack{0 \leq n \leq N \\ 1 \leq k \leq M}} |a_{n,k}|$. Then, we see that
\[ \left| I\left[\prod_{k=1}^{M} w^{r_k}_k \right] \right| \leq \sum_{n=0}^{N} [z^n] \prod_{k=1}^{M} \left( A \sum_{m=0}^{\infty} z^m \right) \leq (A(N+1))^{r_1+\cdots+r_M} \]
for all $r_1,\ldots,r_M \in \mathbb{Z}_{\geq 0}$. This gives $I \in \mathfrak{I}_M$. 
\end{proof}

The inequality \eqref{P_ineq} can be written as
\[ I\left[ \prod_{k=1}^{N} w_k^{r_k} \right](N) \leq I'\left[ \prod_{k=1}^{N} w_k^{r_k} \right](N) \leq I\left[ \prod_{k=1}^{N} (1+w_k)^{r_k} \right](N) \]
and \eqref{goal1} can be written as
\[ I\left[\exp{\left(\sum_{k=1}^{N} \frac{w_k}{k}\right)}\right](N) \leq I'\left[\exp{\left(\sum_{k=1}^{N} \frac{w_k}{k}\right)}\right](N) \leq I\left[\exp{\left(\sum_{k=1}^{N} \frac{1+w_k}{k}\right)}\right](N). \]
Thus, the proof of Theorem \ref{P1} can be interpreted as an application of Lemma \ref{L_I}. 
In practice, it is convenient to calculate $I'$ using the identity $I'[\cdot](N)=J[K[\cdot]]$ where $J:\mathbb{R}\{x\}_{\infty} \to \mathbb{R}, K:\mathbb{R}\{w_1,\ldots,w_N\}_{\infty} \to \mathbb{R}\{x\}_{\infty}$ is defined as
\[ J\left[\sum_{n=0}^{\infty} c_n x^n \right] = \sum_{n=0}^{\infty} \frac{c_n}{n!}, \quad K\left[ f(w_1,\ldots,w_N) \right] = f(Nx/1,Nx/2,\ldots,Nx/N), \]
as we used implicitly in the proof of Theorem \ref{P1}.

By replacing the generating function $\exp{\left(\sum_{k=1}^{N} \frac{w_k}{k}\right)}$, we can produce some other inequalities involving $p(n)$. 
For example, by 
\[ \sum_{n=1}^{\infty} z^{dn} = w_d ,\quad \sum_{n=1}^{\infty} d(n)z^n = \sum_{k=1}^{\infty} w_k, \]
the upper and lower bounds for 
\[ \sum_{n=0}^{N} (N-n) p(n) = \sum_{\substack{m+n \leq N \\ 0<m \\ 0 \leq n}} p(n), \quad \sum_{\substack{n_1 + \cdots + n_r \leq N \\ 0 \leq n_i}} p(n_1) \cdots p(n_r), \quad and \quad \sum_{\substack{n+m \leq N \\ 0<m \\ 0 \leq n }} p(n) d(m) \]
can be obtained by applying this method to 
\[ f(z) = w_1 \exp{\left(\sum_{k=1}^{\infty} \frac{w_k}{k}\right)}, \quad \left( \exp{\left(\sum_{k=1}^{\infty} \frac{w_k}{k}\right)} \right)^r, \quad \left(\sum_{k=1}^{\infty} w_k \right) \exp{\left(\sum_{k=1}^{\infty} \frac{w_k}{k}\right)}. \]
We use this fact to obtain upper and lower bounds for $p(n)$ in the next subsection.

Furthermore, by replacing $I(\cdot)(N)$ and $I'(\cdot)(N)$, we can produce some inequalities similar to Theorem \ref{P1}. 
We use this to estimate the sum of the generalized partition function in subsequent sections.

It might be possible to obtain an explicit formula
\[ \sum_{n=0}^{N} p(n) = I\left[\exp{\left(\sum_{k=1}^{\infty} \frac{w_k}{k}\right)}\right](N) = \sum_{n=1}^{\infty} I_n \left[\exp{\left(\sum_{k=1}^{\infty} \frac{w_k}{k}\right)}\right] \]
if we can establish
\[ I\left[\prod_{k=1}^{M} w^{r_k}_k\right] = \sum_{n=1}^{\infty} I_n\left[\prod_{k=1}^{M} w^{r_k}_{k}\right] \]
for some $I_n:\mathbb{R}[[w_{1},w_{2},w_{3} \cdots]] \to \mathbb{R}$, say, via Poisson summation.

The strategy of comparing linear operators is likely to be useful in other settings as well. 
As an example, we indicate a possible application of this method to the understanding of special functions arising from sieve theory. 
We do not attempt to make the argument fully rigorous here, and we omit the details. In what follows, we give only a rough outline of the argument.

Let $u>0$ and $x>2$. It is well-known that 
\begin{equation} \varPhi(x,x^{1/u}) \coloneqq \#\{ n \leq x : p|n \Rightarrow p>x^{1/u} \} = \frac{x}{\ln{x}}(u \omega(u)+o(1)) \quad (x \to \infty), \end{equation}
where $\omega$ is the Buchstab function. See Lemma 12.1 of \cite{friedlander2025opera} for instance. 

Let $P(s)=\sum_{p>x^{1/u}} p^{-s}$, $P'(s)=\int_{1/u}^{\infty} e^{-sv} dv/v$. Let $I,I':\mathbb{R}\{P\}_{\infty}\setminus \mathbb{R} \to \mathbb{R}$ be linear operators defined by
\[ I\left[\sum_{r=1}^{\infty} c_r P^r\right] = \sum_{1 \leq n \leq x} [n^{-s}] \sum_{r=1}^{\infty} c_r P^r(s), \]
\[ I'\left[\sum_{r=1}^{\infty} c_r P^r\right] = \frac{x}{\ln{x}} \frac{1}{2\pi i} \lim_{T \to \infty} \int_{c-iT}^{c+iT} \left( \sum_{r=1}^{\infty} c_r P'^r(s) \right) e^{s} ds, \]
where $[n^{-s}]F(s)$ denotes the coefficient of $n^{-s}$ in the Dirichlet series $F(s)$. 
Then we see that
\[ \varPhi(x,x^{1/u}) \approx I[e^P-1]. \]
By the prime number theorem, we may see that
\[ I[P^r] \approx \frac{x}{\ln{x}} \int_{\substack{v_1+\cdots+v_r=1 \\ v_i>1/u}} \frac{1}{v_1 \cdots v_r} dv = I'[P^r] \]
for all $r \geq 1$. Thus, one expects
\[ \varPhi(x,x^{1/u}) \approx I'[e^P-1] = \frac{x}{\ln{x}} \frac{1}{2\pi i} \lim_{T \to \infty} \int_{c-iT}^{c+iT} \left(\exp{\left( \int_{1/u}^{\infty} e^{-sv} \frac{dv}{v} \right)}-1\right) e^{s} ds. \]
In fact, it is known that
\[ u \omega(u) = \frac{1}{2\pi i} \lim_{T \to \infty} \int_{c-iT}^{c+iT} \left(\exp{\left( \int_{1/u}^{\infty} e^{-sv} \frac{dv}{v} \right)}-1\right) e^{s} ds. \]
See Section 11.5 of \cite{friedlander2025opera} for instance. 

\subsection{Partition function}

Using \eqref{goal1} and the simple inequality, $p(n) \leq \sum_{m=0}^{n} p (m)$ and $\sum_{m=0}^{n} p (m) \leq 1 + np(n)$, one has
\begin{equation} e^{-H_n} n^{-1} I_0 (2 \sqrt{\zeta_{n}{(2)}n}) - n^{-1} \leq p(n) \leq I_0 (2 \sqrt{\zeta_{n}{(2)}n}) \end{equation}
for all positive integers $n$. 
The upper bound can be improved to $e\zeta_{n}{(2)}^{\frac{1}{2}} n^{-\frac{1}{2}} I_0 (2 \sqrt{\zeta_{n}{(2)}n}) (1 + o(1))$; this follows from a more careful calculation using monotonicity of $p(n)$. 

\begin{comment}
The following holds as $n\to \infty$. 
$$p(n) \leq e\zeta_{n}{(2)}^{\frac{1}{2}} n^{-\frac{1}{2}} I_0 (2 \sqrt{\zeta_{n}{(2)}n}) (1 + o(1))$$
Let $c$ be a constant to be determined later and set $m = \lfloor c \sqrt{n} \rfloor$. By monotonicity of $p(n)$, one has
$$p(n) \leq \frac{1}{m} \left( \sum_{n < i \leq n+m} p(i) \right) \leq \frac{1}{m} \left( \sum_{i \leq n+m} p(i) \right) \leq \frac{1}{m} I_0 (2 \sqrt{\zeta{(2)}(n+m)}).$$
Using the asymptotic formula \eqref{Bessel asy} in conjunction with $\zeta_n{(2)} = \zeta{(2)} + O(n^{-1})$ and $(1+cn^{-\frac{1}{2}})^{1/2} = 1 + \frac{c}{2}n^{-\frac{1}{2}} + O(n^{-1})$, the right-hand side becomes 
\begin{eqnarray*}
\frac{1}{m} I_0 (2 \sqrt{\zeta{(2)}(n+m)}) &=& \frac{\exp{(2 \sqrt{\zeta_n{(2)}(n+m)})}}{m\sqrt{2\pi (2 \sqrt{\zeta_{n}{(2)}n})}}(1+o(1)) \\
&=& \frac{\exp{(2 \sqrt{\zeta_{n}{(2)}n} (1+cn^{-\frac{1}{2}})^{\frac{1}{2}})}}{cn^{\frac{1}{2}}\sqrt{2\pi (2 \sqrt{\zeta_{n}{(2)}n})}}(1+o(1)) \\
&=& \frac{e^{c\sqrt{\zeta{(2)}}}}{c} n^{-\frac{1}{2}} \frac{\exp{(2 \sqrt{\zeta_{n}{(2)}n})}}{\sqrt{2\pi (2 \sqrt{\zeta_{n}{(2)}n})}} (1 + o(1)). 
\end{eqnarray*}
Hence, if we set $c = (\sqrt{\zeta{(2)}})^{-1}$ to minimize $e^{c\sqrt{\zeta{(2)}}}c^{-1}$, we obtain the upper bound since
$$ \frac{\exp{(2 \sqrt{\zeta_{n}{(2)}n})}}{\sqrt{2\pi (2 \sqrt{\zeta_{n}{(2)}n})}} = I_0 (2 \sqrt{\zeta_{n}{(2)}n}) (1 + o(1))$$
by \eqref{Bessel asy}. \par
\end{comment}

It seems difficult to obtain an estimate for $a_n$ for a general sequence $(a_n)_{n \in \mathbb{Z}_{\geq 0}}$ by applying the method used in the proof of Theorem \ref{P1} to
\[ I[\sum_{n=1}^{\infty} (a_{n} - a_{n-1})z^n](N) = a_N - a_0. \]
But for some special cases such as the partition function, we can apply this method with minor modifications. We can actually obtain an estimate,
\[ e^{-H_n} \zeta_{n}{(2)}^{\frac{1}{2}} n^{-\frac{1}{2}} I_1(2\sqrt{\zeta_{n}{(2)} n})(1+o(1)) \leq p(n) \leq \zeta_{n}{(2)}^{\frac{1}{2}} n^{-\frac{1}{2}} I_1(2\sqrt{\zeta_{n}{(2)} n})(1+o(1)). \]

\begin{theorem}
For all $n \in \mathbb{Z}_{\geq 2}$, we have
\[ e^{H_n} p(n) \geq \zeta_{n}{(2)}^{\frac{1}{2}} n^{-\frac{1}{2}} I_1(2\sqrt{\zeta_{n}{(2)} n}) - (\zeta_{n}{(2)}-1)^{\frac{1}{2}} n^{-\frac{1}{2}} I_1(2\sqrt{(\zeta_{n}{(2)} -1)n}) \]
and
\[ p(n) \leq \zeta_{n}{(2)}^{\frac{1}{2}} n^{-\frac{1}{2}} I_1(2\sqrt{\zeta_{n}{(2)} n}) + I_0(2\sqrt{(\zeta_{n}{(2)}-1) n}). \]
\end{theorem}

\begin{proof}
Let $N \in \mathbb{Z}_{\geq 2}$ and $(w_k)$, $I,[\cdot](),I'[\cdot]()$ be defined as before. 

For all $f \in \mathbb{R}[[z]]$, we have
\[ [z^N]f(z) = \sum_{n=1}^{N} ([z^n]-[z^{n-1}]) f(z) + [z^0] f(z) = \sum_{n=0}^{N} [z^n] (1-z) f(z). \]
Since $1+w_1 = \frac{1}{1-z}$ and $w_1 = \frac{z}{1-z}$, this becomes
\[ [z^N]f(z) = \sum_{n=0}^{N} [z^n] \frac{1}{1+w_1} f(z) = \sum_{n=0}^{N-1} [z^n] \frac{1}{w_1} (f(z)-f(0)). \]

Let 
\[ f_M \coloneqq \frac{1}{w_1} \exp{\left(\sum_{k=1}^{M} \frac{w_k}{k}\right)} - \frac{1}{w_1} \exp{\left(\sum_{k=2}^{M} \frac{w_k}{k}\right)} = \left( \sum_{m=1}^{\infty} \frac{w^{m-1}_1}{m!} \right) \prod_{k=2}^{M} \left( \sum_{m=0}^{\infty} \frac{1}{m!}\left( \frac{w_k}{k} \right)^m \right). \]
Note that $f_M \in \mathbb{R}\{w_1,\ldots,w_M\}_{\infty}$ and all of its coefficients of $w_k$'s are positive. For each $M \in \{N,N-1\}$, we have
\[ I\left[ \prod_{k=1}^{M} w_k^{r_k} \right](M) \leq I'\left[ \prod_{k=1}^{M} w_k^{r_k} \right](M) \leq I\left[ \prod_{k=1}^{M} (1+w_k)^{r_k} \right](M) \]
by \eqref{P_ineq}. Thus, we have
\begin{equation} \label{pI1} I[f_{N-1}(w_1,\ldots,w_{N-1}) ](N-1) \leq I'[f_{N-1}(w_1,\ldots,w_{N-1}) ](N-1) \end{equation}
\begin{equation} \label{pI2} I'[f_{N}(w_1,\ldots,w_{N}) ](N) \leq I[f_{N}(1+w_1,\ldots,1+w_{N}) ](N) \end{equation}
by Lemma \ref{L_I}.

For each $M \in \{N,N-1\}$, we have
\begin{align*}
I'[f_{M}(w_1,\ldots,w_{M}) ](M) &= I' \left[ \frac{1}{w_1} \exp{\left(\sum_{k=1}^{M} \frac{w_k}{k}\right)} - \frac{1}{w_1} \exp{\left(\sum_{k=2}^{M} \frac{w_k}{k}\right)} \right](M), \\
&= J\left[ \frac{1}{x} \exp{\left(\sum_{k=1}^{M} \frac{xM}{k^2}\right)} - \frac{1}{x} \exp{\left(\sum_{k=2}^{M} \frac{xM}{k^2}\right)} \right], \\
&= J\left[ \sum_{n=1}^{\infty} \frac{(M\zeta_M(2))^n}{n!} x^{n-1} -  \sum_{n=1}^{\infty} \frac{(M(\zeta_M(2)-1))^n}{n!} x^{n-1} \right], \\
&= \sum_{n=1}^{\infty} \frac{(M\zeta_M(2))^n}{n!(n-1)!} -  \sum_{n=1}^{\infty} \frac{(M(\zeta_M(2)-1))^n}{n!(n-1)!}, \\
&= \zeta_{M}{(2)}^{\frac{1}{2}} M^{-\frac{1}{2}} I_1(2\sqrt{\zeta_{M}{(2)} M}) - (\zeta_{M}{(2)}-1)^{\frac{1}{2}} M^{-\frac{1}{2}} I_1(2\sqrt{(\zeta_{M}{(2)} -1)M}).
\end{align*}
Note that if $M=N-1$, this is bounded by $\zeta_{N}{(2)}^{\frac{1}{2}} N^{-\frac{1}{2}} I_1(2\sqrt{\zeta_{N}{(2)} N})$. The left-hand side of \eqref{pI2} is
\begin{align*}
I[f_{N}(1+w_1,\ldots,1+w_{N}) ](N) &= \sum_{n=0}^{N} [z^n] \left( \frac{e^{H_N}}{1+w_1} \exp{\left(\sum_{k=1}^{N} \frac{w_k}{k}\right)} - \frac{e^{H_{N}-1}}{1+w_1} \exp{\left(\sum_{k=2}^{N} \frac{w_k}{k}\right)} \right), \\
&= [z^N] \left(e^{H_N} \exp{\left(\sum_{k=1}^{N} \frac{w_k}{k}\right)} - e^{H_N-1} \exp{\left(\sum_{k=2}^{N} \frac{w_k}{k}\right)} \right), \\
&\leq e^{H_N} p(N). \\
\end{align*}
This gives the lower bounds for $p(N)$. 

To see the upper bounds, we observe that
\[ I[f_{N-1}(w_1,\ldots,w_{N-1}) ](N-1) = \sum_{n=0}^{N-1} [z^n] f_{N}(w_1,\ldots,w_{N}) = p(N) - p'(N) \]
where 
\[ p'(N) = \sum_{n=0}^{N-1} [z^n] \frac{1}{w_1} \left(\exp{\left(\sum_{k=2}^{N} \frac{w_k}{k}\right)}-1\right)=[z^N] \exp{\left(\sum_{k=2}^{N} \frac{w_k}{k}\right)}. \] 
Applying Lemma \ref{L_I} again, we obtain
\begin{align*}
p'(N) &\leq \sum_{n=0}^{N} [z^n] \exp{\left(\sum_{k=2}^{N} \frac{w_k}{k}\right)}, \\
&= I\left[\exp{\left(\sum_{k=2}^{N} \frac{w_k}{k}\right)}\right](N), \\
&\leq I'\left[\exp{\left(\sum_{k=2}^{N} \frac{w_k}{k}\right)}\right](N) = I_0(2\sqrt{(\zeta_{N}{(2)}-1) N}). 
\end{align*}
This and \eqref{pI1} give the upper bounds for $p(N)$. 
\end{proof}

\section{Preliminaries}

It should be possible to obtain all the specific bounds in this paper by purely elementary means, as in the proof of Theorem \ref{P1}. 
However, to state our results in greater generality, we employ several analytic lemmas.

Throughout the paper, we fix $c>0$. 

Let $\delta{(x)}$ be the function 
\begin{eqnarray*}
\delta{(x)} = 
\begin{cases}
1 & \text{$0 < x$,} \\
\frac{1}{2} & \text{$x=0$,} \\
0 & \text{$x < 0$.}
\end{cases}
\end{eqnarray*}

\begin{lemma}\label{perron}
We have
\begin{equation} \label{perron 1} \frac{1}{2\pi i} \lim_{T \to \infty} \int_{c-iT}^{c+iT} \frac{1}{s} e^{sx} ds = \delta{(x)}, \end{equation}
and
\begin{equation} \label{perron 2} \frac{1}{2\pi i} \int_{c-iT}^{c+iT} \frac{1}{s} e^{sx} ds = O(e^{cx}), \quad (T \geq 2c, x \in \mathbb{R}). \end{equation}
\end{lemma}
See the first lemma in Section 17 of \cite{davenport} for the proof. This is essentially the same as Perron's formula. 
This can be shown by splitting into the cases $x<0$ and $x>0$, closing the contour to the right or to the left accordingly, and applying Cauchy's theorem.

Let us define Wright's generalized Bessel function \cite{wrightBessel} for $z \in \mathbb{C}$ by
\[ {\psi}_u (z) = \sum_{n=0}^{\infty} \frac{z^n}{n!\Gamma{(nu+1)}}. \]

\begin{lemma}\label{psi}
\begin{equation} \label{psi lap} \psi_{u}(a t^{u}) = \frac{1}{2\pi i} \lim_{T \to \infty} \int_{c-iT}^{c+iT} \frac{e^{a s^{-u}}}{s} e^{ts} ds \end{equation}
holds for all $a,t,u,c>0$. 
\end{lemma}
\begin{proof}

Expanding the exponential, we have
\[ \frac{e^{a s^{-u}}}{s} = \sum_{n=0}^{\infty} \frac{1}{n!} a^n s^{-(1+un)}. \]
For each $n \in \mathbb{Z}_{\geq 1}$, we have
\begin{equation}
\int_{c-i\infty}^{c+i\infty} |s|^{-(1+un)} |ds| = c^{-un} \int_{-\infty}^{\infty} \frac{1}{(1+x^2)^{(1+un)/2}} dx \leq c^{-un} C_{u} \label{s_int}
\end{equation}
where $C_{u}=\int_{-\infty}^{\infty} (1+x^2)^{-(1+u)/2} dx$. 
By the dominated convergence theorem and Lemma \ref{perron}, we have
\begin{align*}
\frac{1}{2\pi i} \lim_{T \to \infty} \int_{c-iT}^{c+iT} \frac{e^{a s^{-u}}}{s} e^{ts} ds &= \frac{1}{2\pi i} \lim_{T \to \infty} \int_{c-iT}^{c+iT} \frac{1}{s} e^{ts} ds + \frac{1}{2\pi i} \lim_{T \to \infty} \int_{c-iT}^{c+iT} \sum_{n=1}^{\infty} \frac{1}{n!} a^n s^{-(1+un)} e^{ts} ds, \\
&= 1 + \sum_{n=1}^{\infty} \frac{1}{n!} a^n \frac{1}{2\pi i} \lim_{T \to \infty} \int_{c-iT}^{c+iT} s^{-(1+un)} e^{ts} ds. 
\end{align*}
Using the well-known integral representation of $1/\Gamma$ (or the Laplace inversion theorem) we obtain
\[ \frac{1}{2\pi i} \lim_{T \to \infty} \int_{c-iT}^{c+iT} s^{-(1+un)} e^{ts} ds = \frac{t^{un}}{\Gamma(1+un)}. \]
This completes the proof. 
\end{proof}

\begin{lemma}\label{L_I'}
Let $M \in \mathbb{Z}_{\geq 1}, X \in \mathbb{R}_{>0}$. Let $w_k(s), k=1,\ldots,M$ be functions $\{s \in \mathbb{C} : \Re{s}=c\} \to \mathbb{C}$ such that there exist $\alpha,D>0$ such that
\[ w_k(\overline{s}) = \overline{w_k(s)}, \quad |w_k(s)| \leq D|s|^{-\alpha} \]
holds for all $s\in \mathbb{C}$ with $\Re{s}=c$ and all $k=1,\ldots,M$. Let $I'(\cdot)(N):\mathbb{R}\{w_1,\ldots,w_M\}_{\infty} \to \mathbb{R}$ be given by
\[ I'\left[\sum_{\bm{r} \in \mathbb{Z}^M_{\geq 0}} c_{\bm{r}} \prod_{k=1}^{M} w^{r_k}_k \right](X) = \frac{1}{2\pi i} \lim_{T \to \infty} \int_{c-iT}^{c+iT} \sum_{\bm{r} \in \mathbb{Z}^M_{\geq 0}} c_{\bm{r}} \prod_{k=1}^{M} w^{r_k}_k(s) e^{sX} \frac{ds}{s}. \]
Then $I' \in \mathfrak{I}_M$. 
\end{lemma}

\begin{proof}
If $\bm{r}=0$, we have
\[ \frac{1}{2\pi i} \lim_{T \to \infty} \int_{c-iT}^{c+iT} \prod_{k=1}^{M} w^{r_k}_k(s) e^{sX} \frac{ds}{s} = 1, \]
by Lemma \ref{perron}. Otherwise,
\[ \int_{c-i\infty}^{c+i\infty} \left| \prod_{k=1}^{M} w^{r_k}_k(s) \right| \frac{|ds|}{|s|} \leq D^{r} \int_{c-i\infty}^{c+i\infty} |s|^{-(1+ \alpha r)} |ds| \leq C_{\alpha} \left( \frac{D}{c^{\alpha}} \right)^{r} \]
by \eqref{s_int}, where $r=r_1+\cdots+r_M$, and thus
\[ \sum_{\bm{r} \in \mathbb{Z}^M_{\geq 0} \setminus \{0\}} |c_{\bm{r}}| \int_{c-i\infty}^{c+i\infty} \left| \prod_{k=1}^{M} w^{r_k}_k(s) \right| \frac{|ds|}{|s|} < \infty \]
for all $\sum_{\bm{r} \in \mathbb{Z}^M_{\geq 0}} c_{\bm{r}} \prod_{k=1}^{M} w^{r_k}_k \in \mathbb{R}\{w_1,\ldots,w_M\}_{\infty}$. 
This gives the well-definedness of $I'$ and $I' \in \mathfrak{I}_M$. 
\end{proof}

\section{Generalization 1}\label{Generalization 1}
Let $h$ be a function $\mathbb{R}_{\geq 0} \to \mathbb{R}_{\geq 0}$ satisfying \par
\begin{enumerate}[label=(\roman*)]
\item\label{h_mon} $h(0) = 0$ and $h(x)$ is strictly increasing for $x \geq 0$, \par
\item\label{h_val} $h(n) \in \mathbb{Z}$ for all $n \in \mathbb{Z}_{\geq 0}$. 
\end{enumerate}
Then the generalized partition function $p_h$ is defined by the generating function
\begin{equation} \sum_{n=0}^{\infty} p_{h}(n) z^n = \prod_{n=1}^{\infty} {(1-z^{h(n)} )}^{-1}. \end{equation}
The function $p_h(n)$ is equal to the number of ways to write the integer $n$ as a sum of elements from the set $\{ h(i) : i \in \mathbb{Z}_{\geq 1} \}$ in non-decreasing order. 

Note that \ref{h_mon} and \ref{h_val} imply
\begin{enumerate}[label=(\roman*), resume]
\item\label{h_gr} $\int_{0}^{\infty} e^{-\varepsilon h(x)} dx$ converges for all $\varepsilon>0$. 
\end{enumerate}
and
\[ \sum_{n=0}^{N} p_h (n) = \#\{ (m_i)_{i=1}^{\infty} \in \mathbb{Z}^{\mathbb{N}}_{\geq 0} : \sum_{i=1}^{\infty} m_i h(i) \leq N \} \]
for all $N \in \mathbb{Z}_{\geq 1}$.

\begin{theorem}\label{G1}
Let $h$ be a function $\mathbb{R}_{\geq 0} \to \mathbb{R}_{\geq 0}$ satisfying \ref{h_mon} and \ref{h_gr}. 
Let $X \geq h(1)$ and $M=\lfloor X/h(1) \rfloor$. 
Let $w_k(s)$ and $\lambda(s)$ be defined by
\begin{equation} w_k(s) = \int_{0}^{\infty} e^{-s k h(x)}dx, \qquad \lambda(s) = \sum_{k=1}^{M} \frac{w_k(s)}{k}, \end{equation}
for $s \in \mathbb{C}$ with $\Re{s}=c$. 
Suppose that there exist $\alpha,D>0$ such that $|w_k(s)| \leq D|s|^{-\alpha}$ holds for all $s \in \mathbb{C}$ with $\Re{s}=c$ and all $k=1,\ldots,M$. Let
\begin{equation} \varphi_{h} (x) \coloneqq \frac{1}{2\pi i} \lim_{T \to \infty} \int_{c-iT}^{c+iT} \frac{1}{s} \exp{(\lambda(s))} e^{xs} ds, \quad (x>0) \end{equation}
then one has
\[ e^{-H_M} \varphi_{h} (X) \leq \#\{ (m_i) \in \mathbb{Z}^{\mathbb{N}}_{\geq 0} : \sum_{i=1}^{\infty} m_i h(i) \leq X \}  \leq \varphi_{h} (X). \]
\end{theorem}

\begin{proof}

Fix $X\geq 1$. We begin the proof by showing that there exists $L_0>0$ such that for any $L\geq L_0$ and $(x_i), (a_i) \in \mathbb{Z}_{\geq 0}^{\mathbb{N}}$, the inequality 
\begin{equation} \label{G1_h} \sum_{i=0}^{\infty} a_i h(x_i) \leq X, \end{equation}
is equivalent to 
\begin{equation} \label{G1_hL} \sum_{i=0}^{\infty} a_i \lfloor Lh(x_i) \rfloor \leq \lceil LX \rceil. \end{equation}
\eqref{G1_h} implies \eqref{G1_hL} for any given $L>0$. Thus it suffices to show that \eqref{G1_hL} implies \eqref{G1_h} for some $L>0$. Since the set $\{ \sum_{i=0}^{\infty} a_i h(x_i) : (x_i), (a_i) \in \mathbb{Z}_{\geq 0}^{\mathbb{N}} \}$ has no limit points, there exists $\varepsilon>0$ such that $\sum_{i=0}^{\infty} a_i h(x_i) < X+ \varepsilon$ implies $\sum_{i=0}^{\infty} a_i h(x_i) \leq X$ for all $(x_i), (a_i) \in \mathbb{Z}_{\geq 0}^{\mathbb{N}}$. 
We choose $L>0$ sufficiently large so that $L^{-1} \lceil L X \rceil/\lfloor L h(1) \rfloor < \varepsilon/2, \lceil LX \rceil /L < X+\varepsilon/2$ and assume \eqref{G1_hL}. By \ref{h_mon}, we have
\begin{align*}
\sum_{i=0}^{\infty} a_i \frac{\lfloor Lh(x_i) \rfloor}{L} &\geq \sum_{i=0}^{\infty} a_i h(x_i) - \frac{1}{L} \sum_{\substack{i \in \mathbb{Z}_{\geq 0} \\ h(x_i) \neq 0}} a_i \\
&\geq \sum_{i=0}^{\infty} a_i h(x_i) - \frac{1}{L} \sum_{i=0}^{\infty} a_i \frac{\lfloor Lh(x_i) \rfloor}{\lfloor Lh(1) \rfloor} \geq \sum_{i=0}^{\infty} a_i h(x_i) - \varepsilon/2
\end{align*}
and thus $\sum_{i=0}^{\infty} a_i h(x_i) < X+ \varepsilon$. This implies \eqref{G1_h} by the choice of $\varepsilon$.

We take $L\geq L_0$ sufficiently large so that $\lceil LX \rceil / \lfloor L h(1) \rfloor < \lfloor X/h(1) \rfloor + 1$ holds. 
Let $w_k(z) = \sum_{n=1}^{\infty} z^{k\lfloor Lh(n) \rfloor}$. Let $I,I':\mathbb{R}\{w_1,\ldots,w_M\} \to \mathbb{R}$ be
\[ I\left[\sum_{\bm{r} \in \mathbb{Z}^M_{\geq 0}} c_{\bm{r}} \prod_{k=1}^{M} w^{r_k}_k \right](X) = \sum_{n=0}^{\lceil LX \rceil} [z^n] \sum_{\bm{r} \in \mathbb{Z}^M_{\geq 0}} c_{\bm{r}} \prod_{k=1}^{M} w^{r_k}_k(z), \]
\[ I'\left[\sum_{\bm{r} \in \mathbb{Z}^M_{\geq 0}} c_{\bm{r}} \prod_{k=1}^{M} w^{r_k}_k \right](X) = \frac{1}{2\pi i} \lim_{T \to \infty} \int_{c-iT}^{c+iT} \sum_{\bm{r} \in \mathbb{Z}^M_{\geq 0}} c_{\bm{r}} \prod_{k=1}^{M} w^{r_k}_k(s) e^{sX} \frac{ds}{s}. \]

We have $I \in \mathfrak{I}_M$ by Lemma \ref{L_Id}. By the choice of $L$, we see that $[z^n]w_k(z)=0$ for all $n \leq \lceil LX \rceil$ if $k>M$, and further
\begin{align*}
I\left[ \exp{\left( \sum_{k=1}^{M} \frac{w_k}{k} \right)} \right](X) &= \sum_{n=0}^{\lceil LX \rceil} [z^n] \prod_{i=1}^{\infty} (1 - z^{\lfloor L h(i) \rfloor})^{-1}, \\
&= \sum_{n=0}^{\lceil LX \rceil} [z^n] \prod_{i=1}^{\infty} \sum_{m_i=0}^{\infty} z^{m_i \lfloor L h(i) \rfloor}, \\
&= \#\{ (m_i)_{i} \in \mathbb{Z}^{\mathbb{N}}_{\geq 0} : \sum_{i=1}^{\infty} m_i \lfloor Lh(i) \rfloor \leq \lceil LX \rceil \}, \\
&= \#\{ (m_i)_{i} \in \mathbb{Z}^{\mathbb{N}}_{\geq 0} : \sum_{i=1}^{\infty} m_i h(i) \leq X \}.
\end{align*}

We have $I' \in \mathfrak{I}_M$ by Lemma \ref{L_I'} and 
\[ \varphi_{h} (X) = I'\left[ \exp{\left( \sum_{k=1}^{M} \frac{w_k}{k} \right)} \right](X) \]
from definition. 

Therefore, the desired inequality can be written as follows.
\begin{equation} \label{G1re} I\left[\exp{\left(\sum_{k=1}^{M} \frac{w_k}{k}\right)}\right](X) \leq I'\left[\exp{\left(\sum_{k=1}^{M} \frac{w_k}{k}\right)}\right](X) \leq I\left[\exp{\left(\sum_{k=1}^{M} \frac{1+w_k}{k}\right)}\right](X) \end{equation}

We fix $r_1,\ldots ,r_{M} \in \mathbb{Z}_{\geq 0}$, and show that
\begin{equation} \label{G1_i} I\left[\prod_{k=1}^{M} w^{r_k}_k\right](X) \leq I'\left[\prod_{k=1}^{M} w^{r_k}_k\right](X) \leq I\left[\prod_{k=1}^{M} (1+w_k)^{r_k}\right](X). \end{equation}
This immediately follows from Lemma \ref{perron} if $r_1=\cdots=r_M=0$, thus we may suppose $r \geq 1$.

Let
\[ C = \{ (x^{(k)}_{i})_{\substack{1 \leq k \leq M \\ 1 \leq i \leq r_k}} \in \mathbb{R}^r_{\geq 0} : \sum_{1 \leq k \leq M} \sum_{1\leq i \leq r_k} k h(x^{(k)}_{i}) \leq X \}, \]
\[ C_+ = \{ (x^{(k)}_{i})_{\substack{1 \leq k \leq M \\ 1 \leq i \leq r_k}} \in \mathbb{R}^r_{\geq 0} : \sum_{1 \leq k \leq M} \sum_{1\leq i \leq r_k} k h(\lfloor x^{(k)}_{i} \rfloor) \leq X \}, \]
\[ C_- = \{ (x^{(k)}_{i})_{\substack{1 \leq k \leq M \\ 1 \leq i \leq r_k}} \in \mathbb{R}^r_{\geq 0} : \sum_{1 \leq k \leq M} \sum_{1\leq i \leq r_k} k h(\lceil x^{(k)}_{i} \rceil) \leq X \}. \]
By the monotonicity of $h$, we have $C_{-} \subseteq C \subseteq C_{+}$ and therefore $ \mathrm{vol}(C_{-}) \leq \mathrm{vol}(C) \leq \mathrm{vol}(C_{+})$.

By the choice of $L$, we have
\begin{align*}
I\left[\prod_{k=1}^{M} w_k^{r_k} \right](X) &= \sum_{n=0}^{\lceil LX \rceil} [z^n] \prod_{k=1}^{M} \prod_{i=1}^{r_k} \sum_{x^{(k)}_i=1}^{\infty} z^{k \lfloor L h(x_i) \rfloor} \\
&= \# \{ (x^{(k)}_{i})_{\substack{1 \leq k \leq M \\ 1 \leq i \leq r_k}} \in \mathbb{Z}^r_{\geq 1} : \sum_{1 \leq k \leq M} \sum_{1\leq i \leq r_k} k \lfloor Lh(x^{(k)}_{i}) \rfloor \leq \lceil LX \rceil \}, \\
&= \# \{ (x^{(k)}_{i})_{\substack{1 \leq k \leq M \\ 1 \leq i \leq r_k}} \in \mathbb{Z}^r_{\geq 1} : \sum_{1 \leq k \leq M} \sum_{1\leq i \leq r_k} k h(x^{(k)}_{i}) \leq X \}. 
\end{align*}
and similarly
\begin{align*}
I\left[\prod_{k=1}^{M} (1+w_k)^{r_k} \right](X) &= \sum_{n=0}^{\lceil LX \rceil} [z^n] \prod_{k=1}^{M} \prod_{i=1}^{r_k} \sum_{x^{(k)}_i=0}^{\infty} z^{k \lfloor L h(x_i) \rfloor} \\
&= \# \{ (x^{(k)}_{i})_{\substack{1 \leq k \leq M \\ 1 \leq i \leq r_k}} \in \mathbb{Z}^r_{\geq 0} : \sum_{1 \leq k \leq M} \sum_{1\leq i \leq r_k} k \lfloor Lh(x^{(k)}_{i}) \rfloor \leq \lceil LX \rceil \}, \\
&= \# \{ (x^{(k)}_{i})_{\substack{1 \leq k \leq M \\ 1 \leq i \leq r_k}} \in \mathbb{Z}^r_{\geq 0} : \sum_{1 \leq k \leq M} \sum_{1\leq i \leq r_k} k h(x^{(k)}_{i}) \leq X \}. 
\end{align*}
Thus, we have
\[ I\left[\prod_{k=1}^{M} w_k^{r_k} \right](X) = \mathrm{vol}(C_{-}), \quad I\left[\prod_{k=1}^{M} (1+w_k)^{r_k} \right](X) = \mathrm{vol}(C_{+}). \]

Recall that
\begin{eqnarray*}
\frac{1}{2\pi i} \lim_{T \to \infty} \int_{c-iT}^{c+iT} \frac{1}{s} e^{sx} ds = \delta(x) = 
\begin{cases}
1 & \text{$x > 0$,} \\
\frac{1}{2} & \text{$x=0$,} \\
0 & \text{$x < 0$.}
\end{cases}
\end{eqnarray*}
by Lemma \ref{perron}. Hence, we have
\begin{align*}
I'\left[\prod_{k=1}^{M} w^{r_k}_k\right](X) &= \frac{1}{2\pi i} \lim_{T \to \infty} \int_{c-iT}^{c+iT} \int_{ (x^{(k)}_i) \in \mathbb{R}^r_{\geq 0} } \frac{1}{s} \exp{\left(s\left(X-\sum_{1 \leq k \leq M} \sum_{1\leq i \leq r_k} k h(x^{(k)}_{i})\right)\right)} dx ds \\
&= \frac{1}{2\pi i} \lim_{T \to \infty} \int_{ (x^{(k)}_i) \in \mathbb{R}^r_{\geq 0}} \int_{c-iT}^{c+iT} \frac{1}{s} \exp{\left(s\left(X-\sum_{1 \leq k \leq M} \sum_{1\leq i \leq r_k} k h(x^{(k)}_{i})\right)\right)} ds dx \\
&= \int_{ (x^{(k)}_i) \in \mathbb{R}^r_{\geq 0}} \delta\left(X-\sum_{1 \leq k \leq M} \sum_{1\leq i \leq r_k} k h(x^{(k)}_{i})\right) dx = \mathrm{vol}(C).
\end{align*}
The last equality follows from \ref{h_mon}, and we can justify the interchange of integrals by Fubini's theorem, since
\[ \frac{e^{cX}}{c} \exp{\left(-c\sum_{1 \leq k \leq M} \sum_{1\leq i \leq r_k} k h(x^{(k)}_{i}) \right)} \geq \left| \frac{1}{s} \exp{\left(s\left(X-\sum_{1 \leq k \leq M} \sum_{1\leq i \leq r_k} k h(x^{(k)}_{i})\right)\right)} \right| \]
is integrable over $(x^{(k)}_i) \in \mathbb{R}^r_{\geq 0}, s =c+it, |t| \leq T$ by \ref{h_gr}. 
Also we can justify the interchange of limit and integral by the dominated convergence theorem, since we have
\[ \int_{c-iT}^{c+iT} \frac{1}{s} \exp{\left(s\left(X-\sum_{1 \leq k \leq M} \sum_{1\leq i \leq r_k} k h(x^{(k)}_{i})\right)\right)} ds = O\left(\exp{\left(-c\sum_{1 \leq k \leq M} \sum_{1\leq i \leq r_k} kh(x^{(k)}_{i})\right)}\right) \]
uniformly for all $T \geq 2c$ and $(x^{(k)}_i) \in \mathbb{R}^r_{\geq 0}$, by Lemma \ref{perron}. 

Hence, we have \eqref{G1_i} and therefore \eqref{G1re} by Lemma \ref{L_I}. 
\end{proof}

It is clear from the proof that by including terms with $k > M$, one can also obtain an upper bound where $\lambda$ is replaced with 
$$ \lambda(s) = \sum_{k=1}^{\infty} \frac{w_k(s)}{k} = \sum_{k=1}^{\infty}  \int_{0}^{\infty} \frac{e^{-skh(u)}}{k} du = - \int_{0}^{\infty} \ln{(1-e^{-sh(u)})}du, $$
with some extra condition on $h$. 

It is interesting to note that, although the definition of $p_h$ depends only on the values of $h$ on $\mathbb{Z}_{\geq 1}$, the resulting bounds depend on its behavior on $\mathbb{R}_{\geq 0}$. 
It would be interesting to try to get a better bound on $\sum_{n \leq N} p(n)$, by taking $h$ closer to $\lfloor x \rfloor$ while satisfying \ref{h_mon},\ref{h_val} and $h(n)=n$ for $n \in \mathbb{Z}_{\geq 1}$ (e.g., $h(x)=x-\frac{1}{\pi} (\sin{\pi x})^2$). 
Unfortunately, we do not know such an $h$ for which $\varphi_{h}$ can be computed explicitly.

\begin{corollary}[Partition into $q$-th power]\label{q-th partition}
Let $q>0$ and $X \geq 1$. Then
\[ e^{-H_{\lfloor X\rfloor}} {\psi}_{1/q} (\Gamma{(1+1/q)} \zeta_{\lfloor X \rfloor}{(1+1/q)}X^{\frac{1}{q}}) \leq \#\{ (m_i) \in \mathbb{Z}^{\mathbb{N}}_{\geq 0} : \sum_{i=1}^{\infty} m_i i^{q} \leq X \} \leq {\psi}_{1/q} (\Gamma{(1+1/q)} \zeta_{\lfloor X \rfloor}{(1+1/q)}X^{\frac{1}{q}}). \]
In particular, if $q \in \mathbb{Z}$, we have
\[ e^{-H_N} {\psi}_{1/q} (\Gamma{(1+1/q)} \zeta_{N}{(1+1/q)}N^{\frac{1}{q}}) \leq \sum_{n=0}^{N} p_q (n) \leq {\psi}_{1/q} (\Gamma{(1+1/q)} \zeta_{N}{(1+1/q)}N^{\frac{1}{q}}). \]
for all $N \in \mathbb{Z}_{\geq 1}$. Here, $p_q(n)$ denotes the number of partitions of $n$ into $q$-th power natural numbers. 
\end{corollary}
\begin{proof}
We apply Theorem \ref{G1} with $h(x)= x^q$. We now see that
\[ w_k(s) = \int_{0}^{\infty} e^{-sk x^{q}} dx = \Gamma{(1+1/q)} k^{-\frac{1}{q}} s^{-\frac{1}{q}},\quad \lambda{(s)} = \Gamma{(1+1/q)} \zeta_{\lfloor X \rfloor}{(1+1/q)}s^{-\frac{1}{q}}. \] 
Therefore, by Lemma \ref{psi} $(u,a,t)=(1/q,\Gamma{(1+1/q)} \zeta_{N}{(1+1/q)},X)$, we have 
\[ \varphi_h(X) = \psi_{1/q} (\Gamma{(1+1/q)} \zeta_{\lfloor X \rfloor}{(1+1/q)}X^{\frac{1}{q}}). \]
\end{proof}
Note that by Wright \cite{wrightIII}(stated without proof in \cite{hardy-ramanujan}), the asymptotic behavior of $p_q$ is given by
\[ p_q(n) = n^{-\frac{3}{2}} (2\pi)^{-\frac{q}{2}} \psi^{(-3/2)}_{1/q}{(\Gamma{(1 + 1/q)} \zeta{(1 + 1/q)} n^{1/q})}(1 + o(1)), \]
where 
\[ \psi^{(-3/2)}_{u}{(z)} = \sum_{n=0}^{\infty} \frac{z^n}{n! \Gamma{(nu - 1/2)}}. \]
This, together with the asymptotic formula for $\psi$ (\cite{wrightBessel} Theorem 2), gives
\[ \frac{{\psi}_{1/q} (\Gamma{(1+1/q)} \zeta_{N}{(1+1/q)}N^{\frac{1}{q}})}{\sum_{n=0}^{N} p_q (n)} \asymp N^{\frac{q}{2(1+q)}}, \quad \frac{\sum_{n=0}^{N} p_q (n)}{e^{-H_N} {\psi}_{1/q} (\Gamma{(1+1/q)} \zeta_{N}{(1+1/q)}N^{\frac{1}{q}})} \asymp N^{\frac{2+q}{2(1+q)}} \]
as $N \to \infty$.

\section{Generalization 2}\label{Generalization 2}
Let $g$ be a function $\mathbb{R}_{\geq 0} \to \mathbb{R}$. Then the generalized partition function $p^g$ is defined by
$$ \sum_{n=0}^{\infty} p^{g}(n) z^n = \prod_{n=1}^{\infty} {(1-z^{n} )}^{-g(n)}. $$ 

Let $\mathcal{F}$ denote the set of measurable functions $f:\mathbb{R}_{\geq 0} \to \mathbb{R}_{\geq 0}$ such that $f(x) e^{-\varepsilon x}$ is bounded for every $\varepsilon>0$.  
\begin{theorem}\label{G2}
Let $\epsilon \in \{+,-\}$ and $N \in \mathbb{Z}_{\geq 1}$. Let $g_{\epsilon} \in \mathcal{F}$ be a function such that, for all $x\geq 0$, $g(\lceil x \rceil) \leq g_+(x)$ if $\epsilon=+$ and $g_-(x) \leq g(\lfloor x \rfloor)$ if $\epsilon=-$.

Let $w_k(s), \lambda(s)$ be defined by
\begin{equation} w_k(s) = \int_{0}^{\infty} g_{\epsilon}(x) e^{-s k x}dx, \qquad \lambda(s) = \sum_{k=1}^{N} \frac{w_k(s)}{k}, \end{equation}
for $s \in \mathbb{C}$ with $\Re{s}=c$. 
Suppose that there exist $\alpha,D>0$ such that $|w_k(s)| \leq D|s|^{-\alpha}$ holds for all $s \in \mathbb{C}$ with $\Re{s}=c$ and all $k=1,\ldots,N$. Let
\begin{equation} \varphi^{g_{\epsilon}} (x) \coloneqq \frac{1}{2\pi i} \lim_{T \to \infty} \int_{c-iT}^{c+iT} \frac{1}{s} \exp{(\lambda(s))} e^{xs} ds, \end{equation}
then one has
\[ \sum_{n=0}^{N} p^g (n) \leq \varphi^{g_+} (N) \]
if $\epsilon=+$, and 
\[ e^{-g(0) H_{N}} \varphi^{g_-} (N) \leq \sum_{n=0}^{N} p^g (n) \]
if $\epsilon=-$.  
\end{theorem}

\begin{proof}

Let $w_k(z) = \sum_{n=1}^{\infty} g(n) z^{kn}$. Let $I,I':\mathbb{R}\{w_1,\ldots,w_N\}_{\infty} \to \mathbb{R}$ be
\[ I\left[\sum_{\bm{r} \in \mathbb{Z}^N_{\geq 0}} c_{\bm{r}} \prod_{k=1}^{N} w^{r_k}_k \right](N) = \sum_{n=0}^{N} [z^n] \sum_{\bm{r} \in \mathbb{Z}^N_{\geq 0}} c_{\bm{r}} \prod_{k=1}^{N} w^{r_k}_k(z), \]
\[ I'\left[\sum_{\bm{r} \in \mathbb{Z}^N_{\geq 0}} c_{\bm{r}} \prod_{k=1}^{N} w^{r_k}_k \right](N) = \frac{1}{2\pi i} \lim_{T \to \infty} \int_{c-iT}^{c+iT} \sum_{\bm{r} \in \mathbb{Z}^N_{\geq 0}} c_{\bm{r}} \prod_{k=1}^{N} w^{r_k}_k(s) e^{sN} \frac{ds}{s}. \]
We have $I \in \mathfrak{I}_N$ by Lemma \ref{L_Id} and $I' \in \mathfrak{I}_N$ by Lemma \ref{L_I'}. 
We also have
\[ \sum_{n=0}^{N} p^{g} (n) = I\left[ \exp{\left( \sum_{k=1}^{N} \frac{w_k}{k} \right)} \right](N), \quad \varphi^{g_{\varepsilon}} (N) = I'\left[ \exp{\left( \sum_{k=1}^{N} \frac{w_k}{k} \right)} \right](N). \]

By Lemma \ref{L_I}, it suffices to show that
\begin{equation} \label{G2_i+} I\left[\prod_{k=1}^{N} w^{r_k}_k\right](N) \leq  I'\left[\prod_{k=1}^{N} w^{r_k}_k\right](N), \end{equation}
\begin{equation} \label{G2_i-} I'\left[\prod_{k=1}^{N} w^{r_k}_k\right](N) \leq I\left[\prod_{k=1}^{N} (g(0)+w_k)^{r_k}\right](N), \end{equation}
for all fixed $r_1,\ldots ,r_{N} \in \mathbb{Z}_{\geq 0}$, respectively for each case $\epsilon=+,-$.

We fix $r_1,\ldots ,r_{N} \in \mathbb{Z}_{\geq 0}$ and show \eqref{G2_i+}.
We may suppose $r \geq 1$, otherwise this follows from Lemma \ref{perron}. 
Let $K(x^{(k)}_{i})=\sum_{1 \leq k \leq N} \sum_{1\leq i \leq r_k} k x^{(k)}_{i}$.

Suppose $\epsilon=+$. Then we have
\begin{align*}
I\left[\prod_{k=1}^{N} w^{r_k}_k\right](N) &= \sum_{\substack{(z^{(k)}_i) \in \mathbb{Z}^r_{\geq 1} \\ K(z^{(k)}_i) \leq N}} \prod_{k=1}^{N} \prod_{i=1}^{r_k} g(z^{(k)}_i) \\
&=\sum_{\substack{(z^{(k)}_i) \in \mathbb{Z}^r_{\geq 1} \\ K(z^{(k)}_i) \leq N}} \int_{z^{(k)}_{i} -1 \leq x^{(k)}_{i} \leq z^{(k)}_{i}} \prod_{k=1}^{N} \prod_{i=1}^{r_k} g(\lceil x^{(k)}_i \rceil) dx \\
&\leq \sum_{\substack{(z^{(k)}_i) \in \mathbb{Z}^r_{\geq 1} \\ K(z^{(k)}_i) \leq N}} \int_{z^{(k)}_{i} -1 \leq x^{(k)}_{i} \leq z^{(k)}_{i}} \prod_{k=1}^{N} \prod_{i=1}^{r_k} g_+(x^{(k)}_i) dx \\
&= \int_{ \substack{(x^{(k)}_i) \in \mathbb{R}^{r}_{\geq 0} \\ K(\lceil x^{(k)}_i \rceil) \leq N}} \prod_{k=1}^{N} \prod_{i=1}^{r_k} g_+(x^{(k)}_i) dx \\
&\leq \int_{ \substack{(x^{(k)}_i) \in \mathbb{R}^{r}_{\geq 0} \\ K(x^{(k)}_i) \leq N}} \prod_{k=1}^{N} \prod_{i=1}^{r_k} g_+(x^{(k)}_i) dx. \\
\end{align*}

Suppose $\epsilon=-$. Then 
\begin{align*}
I\left[\prod_{k=1}^{N} (g(0)+w_k)^{r_k} \right](N) &= \sum_{\substack{(z^{(k)}_i) \in \mathbb{Z}^r_{\geq 0} \\ K(z^{(k)}_i) \leq N}} \prod_{k=1}^{N} \prod_{i=1}^{r_k} g(z^{(k)}_i) \\
&=\sum_{\substack{(z^{(k)}_i) \in \mathbb{Z}^r_{\geq 0} \\ K(z^{(k)}_i) \leq N}} \int_{z^{(k)}_{i} \leq x^{(k)}_{i} \leq z^{(k)}_{i}+1} \prod_{k=1}^{N} \prod_{i=1}^{r_k} g(\lfloor x^{(k)}_i \rfloor) dx \\
&\geq \sum_{\substack{(z^{(k)}_i) \in \mathbb{Z}^r_{\geq 0} \\ K(z^{(k)}_i) \leq N}} \int_{z^{(k)}_{i} \leq x^{(k)}_{i} \leq z^{(k)}_{i}+1} \prod_{k=1}^{N} \prod_{i=1}^{r_k} g_-(x^{(k)}_i) dx \\
&= \int_{ \substack{(x^{(k)}_i) \in \mathbb{R}^{r}_{\geq 0} \\ K(\lfloor x^{(k)}_i \rfloor) \leq N}} \prod_{k=1}^{N} \prod_{i=1}^{r_k} g_-(x^{(k)}_i) dx \\
&\geq \int_{ \substack{(x^{(k)}_i) \in \mathbb{R}^{r}_{\geq 0} \\ K(x^{(k)}_i) \leq N}} \prod_{k=1}^{N} \prod_{i=1}^{r_k} g_-(x^{(k)}_i) dx. \\
\end{align*}

Therefore, it remains to show that
\[ I'\left[ \prod_{k=1}^{N} w_k^{r_k} \right](N) = \int_{ \substack{(x^{(k)}_i) \in \mathbb{R}^{r}_{\geq 0} \\ K(x^{(k)}_i) \leq N}} \prod_{k=1}^{N} \prod_{i=1}^{r_k} g_{\epsilon}(x^{(k)}_i) dx. \]
By Lemma \ref{perron}, we have
\begin{align*}
I'\left[\prod_{k=1}^{N} w^{r_k}_k\right](N) &= \frac{1}{2\pi i} \lim_{T \to \infty} \int_{c-iT}^{c+iT} \int_{ (x^{(k)}_i) \in \mathbb{R}^r_{\geq 0} } \frac{1}{s} e^{s(N-K(x^{(k)}_i))} \prod_{k=1}^{N} \prod_{i=1}^{r_k} g_{\epsilon}(x^{(k)}_i) dx ds \\
&= \frac{1}{2\pi i} \lim_{T \to \infty} \int_{ (x^{(k)}_i) \in \mathbb{R}^r_{\geq 0}} \int_{c-iT}^{c+iT} \frac{1}{s} e^{s(N-K(x^{(k)}_i))} \prod_{k=1}^{N} \prod_{i=1}^{r_k} g_{\epsilon}(x^{(k)}_i) ds dx \\
&= \int_{ (x^{(k)}_i) \in \mathbb{R}^r_{\geq 0}} \delta(N-K(x^{(k)}_i)) \prod_{k=1}^{N} \prod_{i=1}^{r_k} g_{\epsilon}(x^{(k)}_i) dx = \int_{\substack{(x^{(k)}_i) \in \mathbb{R}^r_{\geq 0} \\ K(x^{(k)}_i) \leq N}} \prod_{k=1}^{N} \prod_{i=1}^{r_k} g_{\epsilon}(x^{(k)}_i) dx. 
\end{align*}
Here, we can justify the interchanging of the order of integration by Fubini's theorem, and limit and integral by the dominated convergence theorem as in the proof of Theorem \ref{G1}, since $g_{\epsilon} \in \mathcal{F}$ ensures that
\[ \prod_{k=1}^{N} \prod_{i=1}^{r_k} g_{\epsilon}(x^{(k)}_i) \leq C^{r} e^{cK(x^{(k)}_i)/2} \]
for some $C>0$ and all $(x^{(k)}_i) \in \mathbb{R}^{r}_{\geq 0}$. 
\end{proof}

Let $PL(n)$ denote the number of plane partitions with sum $n$. 
For example, there are $6$ plane partitions of $3$ 
\begin{eqnarray*}
3\qquad
\begin{matrix}
2 & 1 \\
\end{matrix}
\qquad
\begin{matrix}
1 & 1 & 1 \\
\end{matrix}
\qquad
\begin{matrix}
1 & 1 \\
1 & \\
\end{matrix}
\qquad
\begin{matrix}
2 \\
1 \\
\end{matrix}
\qquad
\begin{matrix}
1 \\
1 \\
1 \\
\end{matrix}
\end{eqnarray*}
and thus $PL(3)=6$. 

\begin{corollary}[Plane partition]
For all $N \in \mathbb{Z}_{\geq 1}$, one has
\[ e^{-H_N} \psi_{2}{(\zeta_N{(3)} N^2)} (\psi_{1}{(\zeta_N{(2)} N)})^{-1} \leq \sum_{n=0}^{N} PL(n) \leq \psi_{2}{(\zeta_N{(3)} N^2)} \psi_{1}{(\zeta_N{(2)} N)}. \]
\end{corollary}
\begin{proof}
By McMahon's formula, the generating function of $PL(n)$ is given by
$$ \sum_{n=0}^{\infty} PL(n) z^n = \prod_{n=1}^{\infty} {(1-z^{n} )}^{-n}. $$

We apply Theorem \ref{G2} with $g(x) = x+1, g_{-}(x) = x$. We now see that
\[ w_k(s) = s^{-2} k^{-2}, \quad \lambda_{-}(s) = \zeta_N{(3)} s^{-2} \]
and therefore
\[ e^{-H_N} \varphi^{g_-} (N) = e^{-H_N} \psi_{2}{(\zeta_N{(3)} N^2)} \leq \sum_{n=0}^{N} p^{g}(n) \]
by Lemma \ref{psi} $(u,a,t)=(2,\zeta_{N}{(3)},N)$.
On the other hand, the generating function of $p^{g}(n)$ is given by
\[ \prod_{n=1}^{\infty} {(1-z^{n} )}^{-(n+1)} = \left( \prod_{n=1}^{\infty} {(1-z^{n} )}^{-n} \right) \left( \prod_{n=1}^{\infty} {(1-z^{n} )}^{-1} \right) = \sum_{m,n=0}^{\infty} PL(m) p(n) z^{m+n}, \]
and thus
\[ \sum_{n=0}^{N} p^{g}(n) = \sum_{m+n \leq N} PL(m) p(n) \leq \left( \sum_{n=0}^{N} PL(n)\right) \left( \sum_{m=0}^{N} p(m) \right) \leq \psi_{1}{(\zeta_N{(2)} N)} \sum_{n=0}^{N} PL(n). \]
by Theorem \ref{P1}. This gives the lower bound for $\sum_{n=0}^{N} PL(n)$.

We apply Theorem \ref{G2} with $g(x) = x-1, g_{+}(x) = x$. Then 
\[ \varphi^{g_+} (N) = \psi_{2}{(\zeta_N{(3)} N^2)} \geq \sum_{n=0}^{N} p^{g}(n). \]
The generating function of $PL(n)$ is given by
\[ \prod_{n=1}^{\infty} {(1-z^{n} )}^{-n} = \left( \prod_{n=1}^{\infty} {(1-z^{n} )}^{-(n-1)} \right) \left( \prod_{n=1}^{\infty} {(1-z^{n} )}^{-1} \right) = \sum_{m,n=0}^{\infty} p^{g}(m) p(n) z^{m+n} \]
Thus, we have
\[ \sum_{n=0}^{N} PL(n) = \sum_{m+n \leq N} p^{g} (m) p(n) \leq \left( \sum_{m=0}^{N} p^g(m) \right) \left( \sum_{n=0}^{N} p(n) \right) \leq \psi_{2}{(\zeta_N{(3)} N^2)} \psi_{1}{(\zeta_N{(2)} N)} \]
This gives the upper bound. 

\end{proof}
Note that by Wright \cite{wrightI}, the asymptotic behavior of $PL$ is given by
$$ PL(n) = \frac{\zeta{(3)}^{7/36}}{\sqrt{12\pi}} \left( \frac{n}{2} \right)^{-25/36} \exp{\left(3\zeta{(3)}^{1/3} \left( \frac{n}{2} \right)^{2/3} + \zeta'{(-1)} \right)}(1 + o(1)). $$

\section*{Acknowledgments}
The author is deeply grateful to T. Mihara for his valuable comments and corrections, which significantly improved the paper.

\bibliographystyle{plain}
\nocite{*}
\bibliography{ref.bib}

\end{document}